\documentclass[graybox]{svmult}

\usepackage[margin=3.5cm,top=4cm]{geometry}
\usepackage{graphicx}        
\usepackage{multicol}        
\usepackage[bottom]{footmisc}
\usepackage{amssymb}

\begin{document}

\title*{An asymptotic preserving mixed finite element method for wave propagation in pipelines}
\titlerunning{An asymptotic preserving mixed finite element method for wave propagation}
\author{Herbert Egger and Thomas Kugler}
\institute{Herbert Egger \at Technische Universit\"at Darmstadt, Germany, \email{egger@mathematik.tu-darmstadt.de}
\and Thomas Kugler \at Technische Universit\"at Darmstadt, Germany, \email{kugler@mathematik.tu-darmstadt.de}}
\maketitle

\abstract{We consider a parameter dependent family of damped hyperbolic equations with interesting limit behavior: the system approaches steady states exponentially fast and for parameter to zero the solutions converge to that of a parabolic limit problem. We establish sharp estimates and elaborate their dependence on the model parameters. For the numerical approximation we then consider a mixed finite element method in space together with a Runge-Kutta method in time. Due to the variational and dissipative nature of this approximation, the limit behavior of the infinite dimensional level is inherited almost automatically by the discrete problems. The resulting numerical method thus is asymptotic preserving in the parabolic limit and uniformly exponentially stable. These results are further shown to be independent of the discretization parameters. Numerical tests are presented for a simple model problem which illustrate that the derived estimates are sharp in general.}

\section{Introduction}
\label{sec:1}
Pipeline networks in gas or water supply systems are usually made up of rather long pipes and the time scales of interest are typically large as well. The propagation of pressure waves in such long pipes may then be described by a hyperbolic system
\begin{eqnarray}
\partial_t p^\epsilon + \partial_x m^\epsilon &=& 0  \label{eq:dwe1}\\
\epsilon^2 \partial_t m^\epsilon + \partial_x p^\epsilon + a m^\epsilon &=& 0 \label{eq:dwe2}
\end{eqnarray}
together with appropriate initial and boundary conditions. Here $p^\epsilon$ corresponds to the pressure, $m^\epsilon$ to the momentum or mass flux, and $a$ is a generalized friction coefficient which encodes information about the pipe diameter and roughness. This system can be derived by a parabolic rescaling $t = \tilde t \epsilon^2$, $x = \tilde x \epsilon$ of the physical space and time variables $\tilde x, \tilde t$ from the Euler equations or the shallow water equations under some simplifying assumptions \cite{BrouwerGasserHerty11,Osi96} and $\epsilon$ can be assumed to be small.

The parameter dependent hyperbolic problem (\ref{eq:dwe1})--(\ref{eq:dwe2}) has an interesting limit behavior for long time $t \to \infty$ and in the parabolic limit $\epsilon \to 0$ which has been studied intensively in the literature \cite{BrouwerGasserHerty11,GattiPata06,HsiaoLiu92,LopezGomez97,RauchTaylor74,Zuazua88}. Many interesting results are available even for more general problems including the isentropic Euler equations with damping and rather general hyperbolic systems \cite{DafermosPan09,MarRub00}. 
In this note, we contribute to this active research field by establishing the following theoretical results:
\begin{description}[(R1)]
\item[(R1)] For $\epsilon \to 0$, the solutions $(p^\epsilon,m^\epsilon)$ of (\ref{eq:dwe1})--(\ref{eq:dwe2}) converge to the solution $(p^0,m^0)$ of the corresponding parabolic limit problem and 
\begin{eqnarray} \nonumber
\|p^\epsilon(t) - p^0(t)\|^2 + \int_0^t \|m^\epsilon(s) - m^0(s)\|^2 ds \le C \epsilon^2
\end{eqnarray}
with a constant $C$ that is uniform in $\epsilon$ and independent of time $t\ge 0$. 
 \item[(R2)] Assume that the boundary values are kept constant. Then for any $0 \le \epsilon \le 1$ 
the solutions $(p^\epsilon,m^\epsilon)$ converge to the same steady state $(\bar p,\bar m)$ and 
\begin{eqnarray} \nonumber
 \|p^\epsilon(t) - \bar p\|^2 + \epsilon^2 \|m^\epsilon(t) - \bar m\|^2 \le C e^{- \gamma t}
\end{eqnarray}
with constants $C$ and $\gamma>0$ that are independent of $t \ge 0$ and $\epsilon$. 
\end{description}
Our proofs are based on careful energy estimates that explicitly take into account the dependence on the parameter $\epsilon$. As a consequence, the results not only hold for single pipes but can be extended without much difficulty to pipeline networks.

Due to the many important applications, the systematic approximation of parameter dependent hyperbolic problems 
and, in particular, the preservation of asymptotic stability have been investigated intensively as well \cite{BuetDepresFranck12,DuMaTuBe15,ErvedozaZuazua09,GalHerHurLeRou06,GosseToscani02}.
For the discretization of the model problem (\ref{eq:dwe1})--(\ref{eq:dwe2}) we here consider a mixed finite element method in space combined with an implicit Runge-Kutta time-stepping scheme. The resulting method can be shown to exactly conserve mass and to be slightly dissipative in energy, thus capturing the relevant physical behavior \cite{EggerKugler15}. In this paper, we additionally establish the following properties:
\begin{description}[(R1)]
 \item[(R3)] The scheme is \emph{asymptotic preserving}, i.e., the solutions $(p_{h,\tau}^\epsilon,m_{h,\tau}^\epsilon)$ converge with $\epsilon \to 0$ to the solution $(p_{h,\tau}^0,m_{h,\tau}^0)$ of the parabolic limit problem, and 
\begin{eqnarray} \nonumber
 \|p_{h,\tau}^\epsilon(t) - p_{h,\tau}^0(t)\|^2 + \int_0^t \|m_{h,\tau}^\epsilon(s) - m_{h,\tau}^0(s)\|^2 ds \le C \epsilon^2
\end{eqnarray}
with $C$ independent of $\epsilon$ and of the discretization parameters $h$ and $\tau$. 
\item[(R4)] The method is uniformly exponentially stable, i.e., for constant boundary data the solutions $(p_{h,\tau}^\epsilon,m_{h,\tau}^\epsilon)$ converge towards steady state $(\bar p_{h},\bar m_{h})$ and 
\begin{eqnarray} \nonumber
\|p_{h,\tau}^\epsilon(t) - \bar p_{h}\|^2 + \epsilon^2 \|m_{h,\tau}^2(t) - \bar m_{h}\|^2 \le C e^{-\gamma t} 
\end{eqnarray}
with $C$ and $\gamma > 0$ independent of $\epsilon$ and the discretization parameters $h$, $\tau$.
\end{description}
The numerical method is also \emph{well-balanced} in the sense that it automatically provides a stable approximation $(\bar p_{h},\bar m_{h})$ for the corresponding stationary problem.
Since the proposed discretization strategy is of variational and dissipative nature, the above assertions can be proven with only slight modification of the energy arguments used on the continuous level. 
In summary, we thus obtain uniformly stable and accurate approximations for the parameter dependent problem (\ref{eq:dwe1})--(\ref{eq:dwe2}) that capture all relevant physical and mathematical 
properties of the underlying system.

\medskip 

The remainder of this note is organized as follows: 
In Section~\ref{sec:2}, we prove the assertions (R1) and (R2) for the case of a single pipe. 
Section~\ref{sec:3} is then concerned with the numerical approximation 
and the proof of assertions (R3) and (R4) for a single pipe. 
In Section~\ref{sec:4}, we briefly indicate how the results can be generalized with minor modifications to pipe networks. 
In Section~\ref{sec:5}, we discuss in detail a specific test problem and present numerical results 
that illustrate the sharpness of our estimates and also indicate directions for possible improvements.

\section{Analysis on a single pipe} 
\label{sec:2}

Let us start with describing in more detail the model problem under investigation.
The pipe shall be represented by the unit interval and we consider
\begin{eqnarray}
\partial_t p^\epsilon(x,t) + \partial_x m^\epsilon(x,t) &=& 0, \qquad x \in (0,1), \ t>0  \label{eq:sys1}\\
\epsilon^2 \partial_t m^\epsilon(x,t) + \partial_x p^\epsilon(x,t) + a(x) m^\epsilon(x,t) &=& 0, \qquad x \in (0,1), \ t>0. \label{eq:sys2}
\end{eqnarray}
We assume that $0 < \underline a \le a(x) \le \overline a$ and that the pressure at the boundary is given by 
\begin{eqnarray}
p^\epsilon(0,t) = g_0, \qquad p^\epsilon(x,t) &=& g_1, \qquad x \in \{0,1\}, \ t>0. \label{eq:sys3}
\end{eqnarray}
For ease of presentation $g_0$, $g_1$ are assumed to be independent of time here. Other boundary conditions could be considered with obvious modifications. 
From standard results of semigroup theory, one can easily deduce the following.

\begin{lemma} \label{lem:1}
Let $p_0,m_0 \in H^1(0,1)$ be given with $p_0(0)=g_0$ and $p_1(1)=g_1$. 
Then for any $\epsilon>0$ problem (\ref{eq:sys1})--(\ref{eq:sys3}) has a unique classical solution 
\begin{eqnarray}  \nonumber
(p,m) \in C^1(\mathbb{R}^+;L^2(0,1) \times L^2(0,1)) \times C(\mathbb{R}^+;H^1(0,1) \times H^1(0,1)) 
\end{eqnarray}
satisfying initial conditions $p^\epsilon(x,0) = p_0(x)$ and $m^\epsilon(x,0)=m_0(x)$ for all $x \in (0,1)$. 

\noindent
The parabolic problem (\ref{eq:sys1})--(\ref{eq:sys3}) with $\epsilon=0$ 
also has a unique solution 
\begin{eqnarray} \nonumber
p^0 \in C^1(\mathbb{R}^+;L^2(0,1)) \times C(\mathbb{R}^+;H^1(0,1)),
\quad m^0 \in C(\mathbb{R}^+;L^2(0,1))
\end{eqnarray}
satisfying the initial condition $p^0(x,0)=p_0(x)$ for all $x \in (0,1)$. 
\end{lemma}
Note that only one single initial condition is required in the parabolic limit.
By elementary arguments one can verify that the corresponding stationary problem 
\begin{eqnarray} 
\partial_x \bar m(x) &=& 0, \qquad x \in (0,1) \label{eq:sys4}\\
\partial_x \bar p(x) + a(x) \bar m(x) &=& 0, \qquad x \in (0,1) \label{eq:sys5}\\
\bar p(0) = g_0, \quad \bar p(1) &=& g_1 \label{eq:sys6}
\end{eqnarray}
is independent of $\epsilon$ and has a unique solution $(\bar p,\bar m) \in H^1(0,1) \times H^1(0,1)$ as well. 
Using standard energy arguments and the linearity of the time dependent and of the stationary problem, one can then establish the following assertions.
\begin{lemma} \label{lem:2}
Let $(p^\epsilon,m^\epsilon)$ and $(\bar p,\bar m)$ denote solutions of (\ref{eq:sys1})--(\ref{eq:sys3}) and (\ref{eq:sys4})--(\ref{eq:sys6}), respectively. 
Then for any $\epsilon \ge 0$ and any $t \ge 0$, there holds
\begin{eqnarray} 
\|p^\epsilon(t) - \bar p\|^2 + \epsilon^2 \|m^\epsilon(t)-\bar m\|^2 + 2 \int_0^t \underline a \|m^\epsilon(s)-\bar m\|^2 ds \nonumber \\
\le \|p_0-\bar p\|^2 + \epsilon^2 \|m_0-\bar m\|^2.  \nonumber
\end{eqnarray}
For $\epsilon>0$, one can additionally bound the time derivatives of $(p^\epsilon,m^\epsilon)$ by 
\begin{eqnarray} \nonumber 
\|\partial_t p^\epsilon(t)\|^2 + \epsilon^2 \|\partial_t m^\epsilon(t)\|^2 + 2 \int_0^t \underline a \|\partial_t m^\epsilon(s)\|^2 ds \\
\le \|\partial_x m_0\|^2 + \frac{1}{\epsilon^2}\|\partial_x p_0 + a m_0\|^2. \nonumber 
\end{eqnarray}
\end{lemma}
Here and below, $\|\cdot\|$ and $(\cdot,\cdot)$ denote the norm and the scalar product on $L^2(0,1)$. In addition, the functions $p^\epsilon$, $m^\epsilon$ are understood as functions of time with values in Hilbert spaces.
The fact that the second estimate degenerates as $\epsilon \to 0$ resembles the fact that the second initial condition becomes superfluous in the parabolic limit. 
\begin{proof} \smartqed
Due to linearity of the problem, we may assume without loss of generality that $g_0=g_1=0$ and hence $\bar p \equiv \bar m \equiv 0$. From (\ref{eq:sys1})--(\ref{eq:sys2}) we then get 
\begin{eqnarray}
\frac{1}{2} \frac{d}{dt} \|p^\epsilon\|^2 &+&  \frac{\epsilon^2}{2} \frac{d}{dt} \|m^\epsilon\|^2  \nonumber \\
&=& (\partial_t p^\epsilon, p^\epsilon ) + \epsilon^2 (\partial_t m^\epsilon, m^\epsilon) \nonumber \\
&=& -(\partial_x m^\epsilon, p^\epsilon) - (\partial_x p^\epsilon, m^\epsilon) - (a m^\epsilon, m^\epsilon) \nonumber.
\end{eqnarray}
Using integration-by-parts for the second term in the last line, the homogeneous boundary conditions for $p^\epsilon$, and the lower bound for the parameter $a$, we get 
\begin{eqnarray} \nonumber
\frac{d}{dt} \|p^\epsilon\|^2 +  \epsilon^2 \frac{d}{dt} \|m^\epsilon\|^2
\le -2 \underline a \|m^\epsilon\|^2.
\end{eqnarray}
The first estimate now follows by integration with respect to time.
Next assume that $(p^\epsilon,m^\epsilon) \in C^2(\mathbb{R}^+;L^2(0,1) \times L^2(0,1))$. Then by formal differentiation of the problem one can see, that the time derivative
$(\partial_t p^\epsilon, \partial_t m^\epsilon)$ also solves (\ref{eq:sys1})--(\ref{eq:sys3}) 
with homogeneous boundary conditions. The previous estimate thus yields 
\begin{eqnarray} \nonumber
\|\partial_t p^\epsilon(t)\|^2 + \epsilon^2 \|\partial_t m^\epsilon(t)\|^2 + 2\int_0^t \underline a \|\partial_t m^\epsilon(t)\|^2 \\
\le  \|\partial_t p^\epsilon(0)\|^2 + \epsilon^2 \|\partial_t m^\epsilon(0)\|^2. \nonumber
\end{eqnarray}
The differential equations (\ref{eq:sys1}) and (\ref{eq:sys2}) can be used to replace the terms on the right hand side which proves the second estimate for the case of smooth solutions. 
The general case finally follows by a density argument. \qed
\end{proof}

A combination of these energy estimates allows us to provide a precise formulation and to prove the first assertion about solutions of the continuous problem.
\begin{theorem} \label{thm:r1}
Let $\epsilon>0$ and let $(p^\epsilon,m^\epsilon)$ and $(p^0,m^0)$ denote the unique solutions of problem (\ref{eq:sys1})--(\ref{eq:sys3}) with initial values $p^\epsilon(0)=p^0(0)=p^0$ and $m^\epsilon(0)=m_0$. Then 
\begin{eqnarray}
\|p^\epsilon(t) - p^0(t)\|^2 + \int_0^t \underline a \|m^\epsilon(s) - m^0(s)\|^2 ds \nonumber \\
\le \frac{\epsilon^4}{2{\underline a}^2} (\|\partial_x m_0\|^2 + \frac{1}{\epsilon^2} \|\partial_x p_0 + a m_0\|^2). \nonumber
\end{eqnarray}
\end{theorem}
\begin{proof} \smartqed
Let $r^\epsilon = p^\epsilon - p^0$ and $w^\epsilon = m^\epsilon - m^0$ denote the differences between the solutions of the hyperbolic and the parabolic problem. Then by linearity of the equations, one can deduce that $r^\epsilon=0$ at the boundary and that
\begin{eqnarray}
\partial_t r^\epsilon + \partial_x w^\epsilon &=& 0, \nonumber \\
                        \partial_x r^\epsilon + a w^\epsilon &=& - \epsilon^2 \partial_t m^\epsilon. \nonumber
\end{eqnarray}
Applying similar arguments as in the proof of the previous lemma then leads to
\begin{eqnarray}
\frac{1}{2} \frac{d}{dt}\|r^\epsilon(t)\|^2 + \underline a \|w^\epsilon(t)\|^2 
&\le& \epsilon^2 \|\partial_t m^\epsilon(t)\| \|w^\epsilon(t)\| \nonumber \\
&\le& \frac{\epsilon^4}{2 \underline a} \|\partial_t m^\epsilon(t)\|^2 + \frac{\underline a}{2} \|w^\epsilon(t)\|^2. \nonumber
\end{eqnarray}
Multiplication by two and integration with respect to time further yields
\begin{eqnarray} \nonumber
\|r^\epsilon(t)\|^2 + \int_0^t \underline a \|w^\epsilon(s)\|^2 ds 
\le \|r^\epsilon(0)\|^2 + \frac{\epsilon^4}{\underline a} \int_0^t \|\partial_t m^\epsilon(s)\|^2 ds. 
\end{eqnarray}
Since $p^\epsilon$ and $p^0$ satisfy the same initial conditions, we have $r^\epsilon(0)=0$,
and the remaining integral on the right hand side can be estimated by Lemma~\ref{lem:2}. \qed
\end{proof}

The estimates of Lemma~\ref{lem:2} provide uniform bounds for the distance to steady state.
A refined analysis reveals that in fact exponential convergence takes place.
\begin{theorem} \label{thm:r2}
Let $(p^\epsilon,m^\epsilon)$ denote a solution of (\ref{eq:sys1})--(\ref{eq:sys3}) for some $0 \le \epsilon \le 1$. Further let $(\bar p,\bar m)$ be the unique solution of the corresponding stationary problem. Then 
\begin{eqnarray} \nonumber
\|p^\epsilon(t) - \bar p\|^2 + \epsilon^2 \|m^\epsilon(t)-\bar m\|^2 \le C e^{-\gamma (t-s)}  ( \|p^\epsilon(s) - \bar p\|^2 + \epsilon^2 \|m^\epsilon(s)-\bar m\|^2)
\end{eqnarray}
which holds for all $0 \le s \le t$ and with some constants $C,\gamma > 0$ independent of $\epsilon$.
\end{theorem}
\begin{proof} \smartqed
Set $\tau = t/\epsilon$ and $\sigma=s/\epsilon$ and define $\pi^\epsilon(\tau) = p^\epsilon(t)$ and $\mu^\epsilon(\tau)=\epsilon m^\epsilon(t)$. Then by elementary calculations, one can see that
\begin{eqnarray}
\partial_t \pi^\epsilon + \partial_x \mu^\epsilon &=& 0  \nonumber\\
\partial_t \mu^\epsilon + \partial_x \pi^\epsilon + \frac{a}{\epsilon} \mu^\epsilon &=& 0. \nonumber
\end{eqnarray}
The exponential convergence for this problem has been established in \cite{EggerKugler15} and a direct application of Theorem~3.3 in \cite{EggerKugler15} yields 
$$
\|\pi^\epsilon(\tau) - \bar \pi\|^2 + \|\mu^\epsilon(\tau) - \bar \mu\|^2 \le C e^{-c \epsilon (\tau-\sigma)} (\|\pi^\epsilon(\sigma) - \bar \pi\|^2 + \|\mu^\epsilon(\sigma) - \bar \mu\|^2).
$$ 
Using $\tau = t/\epsilon$ and $\sigma=s/\epsilon$  and the definition of $\pi^\epsilon$ and $\mu^\epsilon$ then directly yields the estimate for $\epsilon>0$.
The result for $\epsilon = 0$ follows directly but also from the uniformity of those for $\epsilon>0$ and the convergence to the parabolic limit. \qed
\end{proof}

\section{A mixed finite element Runge-Kutta scheme}
\label{sec:3}

For the discretization of problem (\ref{eq:sys1})--(\ref{eq:sys3}), we now consider a mixed finite element method in space and the implicit Euler method in time. More general Galerkin and time-integration schemes could be analyzed in a similar manner. Let $T_h=\{e\}$ denote a uniform mesh of the interval $(0,1)$ into elements $e$ of size $h$ and denote by 
\begin{eqnarray} \nonumber
Q_h = \{ q \in L^2(0,1) : q|_e \in P_0(e)\} 
\quad \mbox{and} \quad 
V_h = \{v \in C[0,1] : v|_e \in P_1(e) \} 
\end{eqnarray}
the spaces of piecewise constant and piecewise linear and continuous functions, respectively.
Furthermore, let $\tau>0$ be the time step size, define $t^k = k \tau$, and denote by $\bar \partial_\tau u(t^k) = \frac{1}{\tau} [u(t^k) - u(t^{k-1})]$ the backward difference quotient. 
We then consider 
\begin{problem} \label{prob:1}
Let $p^\epsilon_{h,\tau}(0)$ and $m^\epsilon_{h,\tau}(0)$ be the $L^2$ projections of the initial data onto the finite element spaces. For $k \ge 1$ find $(p_{h,\tau}^\epsilon(t^k),m_{h,\tau}^\epsilon(t^k)) \in Q_h \times V_h$, such that 
\begin{eqnarray}
(\bar \partial_\tau p_{h,\tau}^\epsilon(t^k), q_h) +  (\partial_x m_{h,\tau}^\epsilon(t^k), q_h) &=& 0 \nonumber \\
\epsilon^2 (\bar \partial_\tau m_{h,\tau}^\epsilon(t^k), v_h) -  (p_{h,\tau}^\epsilon(t^k), \partial_x v_h) + (a  m_{h,\tau}^\epsilon(t^k), v_h) &=& g_0 v_h(0) - g_1 v_h(1) \nonumber
\end{eqnarray}
holds for all test functions $q_h \in Q_h$ and all $v_h \in V_h$.
\end{problem}
Recall that $(\cdot,\cdot)$ denotes the scalar product of $L^2(0,1)$. 
Existence of a unique discrete solution $(p^\epsilon_{h,\tau},m^\epsilon_{h,\tau})$ to Problem~\ref{prob:1} and of a unique solution $(\bar p_h,\bar m_h)$ of the corresponding stationary problem can be deduced from the results in \cite{EggerKugler15}. 
\begin{lemma} \label{lem:3}
For any $\epsilon \ge 0$, Problem~\ref{prob:1} admits a unique solution $(p_{h,\tau}^\epsilon,m_{h,\tau}^\epsilon)$ and
\begin{eqnarray}
\|p^\epsilon_ {h,\tau}(t^k) - \bar p_h\|^2 + \epsilon^2  \|m^\epsilon_{h,\tau}(t^k) - \bar m_h\|^2
+ 2 \underline a \sum_{j=1}^k \tau \|m^\epsilon_{h,\tau}(t^j) - \bar m_h\|^2 \nonumber \\
\le \|p_0 - \bar p_h\|^2 + \epsilon^2  \|m_0 - \bar m_h\|^2 \nonumber
\end{eqnarray} 
for all $k \ge 0$, where $(\bar p_h,\bar m_h) \in Q_h \times V_h$ denotes the unique solution of the corresponding stationary problem. 
For $\epsilon>0$, we additionally have 
\begin{eqnarray}
\|\bar \partial_\tau p^\epsilon_{h,\tau}(t^k)\|^2 + \epsilon^2  \|\bar \partial_\tau m^\epsilon_{h,\tau}(t^k)\|^2
+ 2 \underline a \sum_{j=1}^k \tau \|\bar \partial_\tau m^\epsilon_{h,\tau}(t^j)\|^2 \nonumber \\
\le C (\|\partial_x m_0\|^2 + \frac{1}{\epsilon^2}\|\partial_x p_0 + a m_0\|^2 + \overline a^2 \|m_0\|^2) \nonumber 
\end{eqnarray} 
with constant $C$ that is independent of $\epsilon$ and the discretization parameters $h$ and $\tau$.
\end{lemma}
\begin{proof} \smartqed
Without loss of generality, we may set $g_0=g_1=0$ and hence $\bar p_h \equiv \bar m_h \equiv 0$. 
For ease of notation, let us abbreviate $p^{k} := p_{h,\tau}^\epsilon(t^k)$ and $m^k := m_{h,\tau}^\epsilon(t^k)$. 
Then by elementary calculations, one can verify that
\begin{eqnarray}
\|p^{k}\|^2 + \epsilon^2 \|m^k\|^2 &+& \|p^k-p^{k-1}\|^2 + \epsilon^2 \|m^k-m^{k-1}\|^2 \nonumber\\
&=& \|p^{k-1}\|^2 + \epsilon^2 \|m^{k-1}\|^2 + 2\tau [(\bar \partial_\tau p^k,p^k ) + \epsilon^2 (\bar \partial_\tau m^k, m^k)]. \nonumber
\end{eqnarray}
Using the discrete problem and the lower bounds for the parameter, we thus obtain 
\begin{eqnarray} \nonumber
\|p^{k}\|^2 + \epsilon^2 \|m^k\|^2
\le \|p^{k-1}\|^2 + \epsilon^2 \|m^{k-1}\|^2 - 2 \underline a \tau \|m^k\|^2. 
\end{eqnarray}
The first estimate now follows by recursion and by noting that $\|p^0\| \le \|p_0\|$ and $\|m^0\| \le \|m_0\|$, since the initial iterates were defined as $L^2$ orthogonal projections of the initial values onto the respective subspaces. 
By linearity of the problem, one can then deduce in a similar manner that 
\begin{eqnarray} \nonumber
\|\bar \partial_\tau p^{k}\|^2 + \epsilon^2 \|\bar \partial_\tau m^k\|^2
+ 2 \underline a \sum_{j=2}^k \tau \|\bar \partial_\tau m^j\|^2 
\le \|\bar \partial_\tau p^{1}\|^2 + \epsilon^2 \|\bar \partial_\tau m^1\|^2.
\end{eqnarray}
Using the discrete problem for $k=1$, we further get 
\begin{eqnarray}
&&\tau (\|\bar \partial_\tau p^1\|^2 + \epsilon^2 \|\bar \partial_\tau m^1\|^2) \nonumber \\
&&= -(\partial_x m^1, p^1 - p^0) + (p^1, \partial_x m^1 - \partial_x m^0) - (a m^1, m^1-m^0) \nonumber \\
&&\le -(m^1-m^0, \partial_x p_0 + a m^0) - (p^1 - p^0, \partial_x m^0) \nonumber \\
&&\le \tau \|\bar \partial_\tau m^1 \| \|\partial_x p_0 + a m^0\| + \tau \|\bar \partial_\tau p^1\| \|\partial_x m^0\| . \nonumber
\end{eqnarray}
Using Young's inequality, the bounds for the parameter $a$, and the stability of the $L^2$ projection in the $H^1$ norm, we may conclude that
\begin{eqnarray}
\|\bar \partial_\tau p^1\|^2 + \epsilon^2 \|\bar \partial_\tau m^1\|^2 
\le C' \|\partial_x m_0\|^2 + \frac{1}{\epsilon^2} (2\|\partial_x p_0 + a m_0\|^2 + 2 \overline a^2 \|m^0 - m_0\|^2), \nonumber
\end{eqnarray}
which together with the energy estimate from above completes the proof. \qed
\end{proof}

Similarly as on the continuous level, a combination of the previous estimates now immediately allows to show convergence of the solutions $(p^\epsilon_{h,\tau},m^\epsilon_{h,\tau})$ of the discrete hyperbolic problem to that of the discrete parabolic problem when $\epsilon \to 0$. 
\begin{theorem} \label{thm:r3}
Let $(p^\epsilon_{h,\tau},m^\epsilon_{h,\tau})$ and $(p^0_{h,\tau},m^0_{h,\tau})$ denote solutions of Problem~\ref{prob:1} for $\epsilon>0$ and $\epsilon=0$, respectively. Further assume that $p^\epsilon_{h,\tau}(0)=p^0_{h,\tau}(0)$. Then 
\begin{eqnarray} 
\|p_{h,\tau}^\epsilon(t^k) - p_{h,\tau}^0(t^k)\|^2 + 2 \underline a \sum_{j=1}^k \tau \|m^\epsilon_{h,\tau}(t^j) - m^0_{h,\tau}(t^j)\|^2 \nonumber \\
\le C \epsilon^4 (\|\partial_x m_0\|^2 + \frac{1}{\epsilon^2} \|\partial_x p_0 + a m_0\|^2 + \frac{\overline a^2}{\epsilon^2} \|m_0\|^2) \nonumber
\end{eqnarray}
with constant $C$ independent of $\epsilon$ and of the discretization parameters $h$ and $\tau$.
\end{theorem}
\begin{proof} \smartqed
Define $r^k = p_{h,\tau}^\epsilon(t^k) - p_{h,\tau}^0(t^k)$ and $w^k = m_{h,\tau}^\epsilon(t^k) - w^0_{h,\tau}(t^k)$. 
Then by linearity of the discrete problem, one can see that 
\begin{eqnarray}
(\bar \partial_\tau r^k, q_h) + (\partial_x w^k, q_h) &=&0 \nonumber \\
- (r^k, \partial_x v_h) + (a w^k, v_h) &=& -\epsilon^2 (\bar \partial_\tau m^\epsilon_{h,\tau}(t^k), v_h)  \nonumber 
\end{eqnarray}
for all $q_h \in Q_h$ and $v_h \in V_h$ and for all $k \ge 0$. 
Testing with $q_h = w^k$ and $v_h = m^k$ and proceeding similarly as in the previous lemmas leads to 
the energy estimate
\begin{eqnarray}
\|r^k\|^2 + 2 \underline a \sum_{j=1}^k \tau \|w^k\|^2  \nonumber 
\le \|r^0\|^2 + \epsilon^2 \sum_{j=1}^k \tau \|\bar \partial_\tau m^\epsilon_{h,\tau}(t^j)\| \|w^k\| \nonumber \\
\le \|r^0\|^2 + \underline a \sum_{j=1}^k \tau \|w^k\|^2 + \frac{\epsilon^4}{\underline a} \sum_{j=1}^k \tau \|\bar \partial_\tau m^\epsilon_{h,\tau}(t^j)\|^2.\nonumber 
\end{eqnarray}
The assertion now follows by noting that $r^0 \equiv 0$ and application of the second estimate of the previous lemma to estimate the last term in this expression. \qed
\end{proof}

Similarly as on the continuous level, one can again prove uniform exponential convergence of discrete solutions to steady states.
\begin{theorem} \label{thm:r4}
 Let $(p^\epsilon_{h,\tau},m^\epsilon_{h,\tau})$ denote a solution of Problem~\ref{prob:1} and let $(\bar p_h,\bar m_h)$ let be the unique solution of the 
 corresponding stationary problem. Then
 \begin{eqnarray}
  &\|p_{h,\tau}^\epsilon(t^k)-\bar p_h\|^2 &+ \epsilon^2\|m_{h,\tau}^\epsilon(t^k)-\bar m_h\|^2 \nonumber\\
 &&\le C e^{-\gamma(k-j)\tau}\|p_{h,\tau}^\epsilon(t^j)-\bar p_h\|^2 + \epsilon^2\|m_{h,\tau}^\epsilon(t^j)-\bar m_h\|^2\nonumber
 \end{eqnarray}
for all $0\le j\le k$ with constants $C,\gamma>0$ that are independent of $\epsilon$, $h$, and $\tau$.
\end{theorem}
\begin{proof} \smartqed
Using a rescaling like in the proof of Theorem~\ref{thm:r2}, the result for $\epsilon>0$ can be deduced directly from Theorem~7.4 in \cite{EggerKugler15}. The estimate for $\epsilon=0$ follows from the uniformity of the estimates and convergence to the parabolic limit.\qed
\end{proof}

\section{Extension to pipe networks}
\label{sec:4}

The results of the previous sections can be extended to the following class of hyperbolic problems on networks: Let $\mathcal{G}=(\mathcal{V},\mathcal{E})$ be a finite directed graph representing the 
topology of the network. On every single pipe $e$, the dynamics shall again be described by the 
linear damped hyperbolic system
\begin{eqnarray}
\partial_t p^\epsilon_e + \partial_x m^\epsilon_e &=& 0 \label{eq:net1}\\
\epsilon^2 \partial_t m^\epsilon_e + \partial_x p^\epsilon_e + a_e m^\epsilon_e &=& 0. \label{eq:net2}
\end{eqnarray}
At any junction $v$ of several pipes $e \in \mathcal{E}(v)$ of the network, we require that 
\begin{eqnarray}
\sum_{e \in \mathcal{E}(v)} n_e(v) m^\epsilon_e(v) = 0 \label{eq:net3} \\
p^\epsilon_e(v) = p_v \qquad \forall e \in \mathcal{E}(v). \label{eq:net4}
\end{eqnarray}
Here $n_e(v)$ takes the value minus or plus one, depending on whether the pipe $e$ start or ends at the junction $v$.
At the boundary vertices $v$ of the network, we require
\begin{eqnarray}
p^\epsilon_e(v)=g_v.  \label{eq:net5}
\end{eqnarray}
Using the arguments developed in \cite{EggerKugler16}, all results stated in Theorem~\ref{thm:r1}--\ref{thm:r4}
hold verbatim also for the system (\ref{eq:net1})--(\ref{eq:net5}). Details are left to the interested reader.

\section{Numerical validation}
\label{sec:5}
We now illustrate our theoretical results by considering in detail a particular model problem. 
For constant damping parameter $a\equiv1$, initial data $p_0=sin(\pi x)$, $m_0 \equiv 0$, and boundary values $g_0=g_1\equiv0$, the solution of problem (\ref{eq:sys1})--(\ref{eq:sys3}) is given by
\begin{eqnarray}
p^\epsilon(x,t) &=&\left(\frac{2\pi^2\epsilon^2}{1-s(\epsilon)}\frac{1}{s(\epsilon)}\exp{\left(-\frac{1}{2\epsilon^2}(1-s(\epsilon))t\right)}\right. \nonumber\\
		  &&- \left.\frac{2\pi^2\epsilon^2}{1+s(\epsilon)}\frac{1}{s(\epsilon)}\exp{\left(-\frac{1}{2\epsilon^2}(1+s(\epsilon))t\right)}\right)\sin(\pi x)\nonumber
\end{eqnarray}
and 
\begin{eqnarray}
 m^\epsilon(x,t) &=& \left(\frac{\pi}{s(\epsilon)}\exp{\left(-\frac{1}{2\epsilon^2}(1-s(\epsilon))t\right)}\right.\nonumber\\
		  &&-\left.\frac{\pi}{s(\epsilon)}\exp{\left(-\frac{1}{2\epsilon^2}(1+s(\epsilon))t\right)}\right) \cos(\pi x) \nonumber.
\end{eqnarray}
with parameter $s(\epsilon)=\sqrt{1-4\pi^2\epsilon^2}$. 
By Taylor expansion w.r.t. $\epsilon$, we deduce that
\begin{eqnarray}
 p^\epsilon(x,t) &=& \left((1+\mathcal{O}(\epsilon^2))\exp{\biggl((-\pi^2-\mathcal{O}(\epsilon^2))t\biggr)}\right.\nonumber\\
		  &&-\left.\mathcal{O}(\epsilon^2)\exp{\left((-\frac{1}{\epsilon^2}+\mathcal{O}(1))t\right)}\right)\sin(\pi x)\nonumber
\end{eqnarray}
and
\begin{eqnarray}
 m^\epsilon(x,t) &=& \left((\pi + \mathcal{O}(\epsilon^2))\exp{\biggl((-\pi^2-\mathcal{O}(\epsilon^2))t\biggr)}\right.\nonumber\\
	      &&-\left.(\pi + \mathcal{O}(\epsilon^2))\exp{\left((-\frac{1}{\epsilon^2}+\mathcal{O}(1))t)\right)}\right)\cos(\pi x). \nonumber
\end{eqnarray}
For $\epsilon=0$, we simply obtain $p^0(x,t)=e^{-\pi^2 t} \sin(\pi x)$ and $m^0(x,t)=\pi e^{-\pi^2 t} \cos(\pi x)$
and the steady state for this problem is given by $\bar p,\bar m\equiv0$. 

From the explicit solution formulas, one can then immediately see that exponential convergence towards the steady state takes place  with $t \to \infty$ for all $0 \le \epsilon \le 1$ with a rate that is independent of $\epsilon$ which was the assertion of Theorem \ref{thm:r2}.
In Table~\ref{table:r4}, we depict numerical results obtained with the numerical scheme discussed in Section~\ref{sec:3}. 
\begin{table}[ht!]
\renewcommand{\arraystretch}{1.2}
\setlength\tabcolsep{0.5em}
\begin{center}
\begin{tabular}{c||c|c|c|c|c|c}
$t\backslash \epsilon$  & 1/4 & 1/8 & 1/16 & 1/32 & 1/64 & 1/128\\
\hline
\hline
     0.0 & 5.00e-01 & 5.00e-01 & 5.00e-01 & 5.00e-01 & 5.00e-01 & 5.00e-01 \\ \hline
     0.1 & 2.72e-01 & 9.09e-02 & 7.26e-02 & 7.02e-02 & 6.96e-02 & 6.95e-02 \\ \hline
     0.5 & 3.56e-04 & 5.35e-06 & 1.94e-05 & 2.42e-05 & 2.54e-05 & 2.57e-05 \\ \hline
     1.0 & 8.51e-08 & 2.71e-11 & 6.64e-10 & 1.13e-09 & 1.28e-09 & 1.32e-09  \\ \hline
\hline
$\gamma$ & 15.59    & 23.64    & 20.44    & 19.90    & 19.78    & 19.75

\end{tabular} 
\smallskip
\caption{Distance $\|p_{h,\tau}^\epsilon(t) - \bar p_h\|^2 + \epsilon^2 \|m_{h,\tau}^\epsilon - \bar m_h\|^2 \le C e^{-\gamma t}$ of the numerical solution to the discrete steady state for different values of $\epsilon$ and times $t=0,0.1,0.5,1.0$ and estimated exponential convergence rate $\gamma$. Discretization parameters were set to $h=0.01$ and $\tau=10^{-5}$. }
 \label{table:r4}
\end{center}
\end{table}
As predicted by Theorem~\ref{thm:r4}, the exponential convergence towards steady state with $t \to \infty$ is uniform in $\epsilon$ also for the discrete schemes. Mesh independence of 
the exponential decay rate was already demonstrated in \cite{EggerKugler16}. 

Let us next have a closer look on the convergence to the parabolic limit. 
Using the analytical solution formulas and Taylor expansion w.r.t. $\epsilon$, one can deduce that 
\begin{eqnarray}
 p^\epsilon-p^0 &=& \left(\mathcal{O}(\epsilon^2)(t+1)\exp{\biggl(-\pi^2t+\mathcal{O}(1)t\biggr)}\right.\nonumber\\ 
		&&+ \left.\mathcal{O}(\epsilon^2)\exp{\left(-\frac{1}{\epsilon^2}t+\mathcal{O}(1)t\right)}\right)\sin(\pi x)\nonumber
\end{eqnarray}
and
\begin{eqnarray}
 m^\epsilon-m^0 &=& \left(\mathcal{O}(\epsilon^2)(t+1)\exp{\biggl(-\pi^2t+\mathcal{O}(1)t\biggr)}\right.\nonumber\\
		&&-\left.(\pi+\mathcal{O}(\epsilon^2))\exp{\left(-\frac{1}{\epsilon^2}t+\mathcal{O}(1)t\right)}\right)\cos(\pi x).\nonumber
\end{eqnarray}
This shows that $\|p^\epsilon-p^0\|^2 = \mathcal{O}(\epsilon^4)$ and $\int_0^t\| m^\epsilon-m^0\|^2 =\mathcal{O}(\epsilon^2)$ which yields exactly the asymptotic behavior predicted in Theorem~\ref{thm:r1}.
In Table~\ref{table:r3,1}, we display the corresponding results obtained with the proposed discretization scheme. 
\begin{table}[ht!]
\renewcommand{\arraystretch}{1.2}
\setlength\tabcolsep{0.5em}
\begin{center}
\begin{tabular}{c||c|c|c|c|c|c||c}
$t^k\backslash \epsilon$ & 1/4 & 1/8 & 1/16 & 1/32 & 1/64 & 1/128 & $\alpha$ \\
\hline
\hline
0.1  & 9.81e-02 & 3.47e-02 & 9.41e-03 & 2.38e-03 & 5.89e-04 & 1.39e-04 & 1.87 \\ \hline
0.5  & 1.18e-01 & 3.58e-02 & 9.44e-03 & 2.39e-03 & 5.89e-04 & 1.39e-04 & 1.93 \\ \hline
1.0  & 1.18e-01 & 3.58e-02 & 9.44e-03 & 2.39e-03 & 5.89e-04 & 1.39e-04 & 1.93 \\
\end{tabular}
\smallskip
\caption{Error $\|p_{h,\tau}^\epsilon(t^k) - p_{h,\tau}^0(t^k)\|^2 + \sum_{j=1}^k \underline a \|m_{h,\tau}^\epsilon(t^j) - m_{h,\tau}^0(t^j)\|^2 = O(\epsilon^\alpha)$ between the discrete approximations for the hyperbolic problem and the parabolic limit problem for different values of $\epsilon$ and time steps $t^k$ and observed convergence rate $\alpha$. Discretization with $h=0.01$ and $\tau=10^{-5}$.}
 \label{table:r3,1}
\end{center}
\end{table}
Also here we can exactly observe the convergence rate predicted by Theorem~\ref{thm:r3}. 
Note that the second term in the error measure is strictly increasing w.r.t. time, 
which together with the exponential convergence to steady states explains that the error is almost independent of $t$ here. 

In Table~\ref{table:r3,2}, we report about further numerical tests to illustrate the independence of the results on the discretization parameters. 
\begin{table}[ht!]
\renewcommand{\arraystretch}{1.2}
\setlength\tabcolsep{0.5em}
\begin{center}
\begin{tabular}{c||c|c|c|c|c|c||c}
$\epsilon$ & 1/4 & 1/8 & 1/16 & 1/32 & 1/64 & 1/128 & $\alpha$\\
\hline
\hline
$h=0.010,\tau=10^{-5}$  & 1.18e-01 & 3.58e-02 & 9.44e-03 & 2.39e-03 & 5.89e-04 & 1.39e-04 & 1.93 \\ \hline
$h=0.002,\tau=10^{-5}$  & 1.18e-01 & 3.58e-02 & 9.44e-03 & 2.39e-03 & 5.89e-04 & 1.39e-04 & 1.99 \\ \hline
$h=0.010,\tau=10^{-6}$  & 1.18e-01 & 3.58e-02 & 9.46e-03 & 2.40e-03 & 6.00e-04 & 1.49e-04 & 1.90 \\ \hline
$h=0.002,\tau=10^{-6}$  & 1.18e-01 & 3.58e-02 & 9.46e-03 & 2.40e-03 & 6.00e-04 & 1.49e-04 & 1.90 \\
\end{tabular}
\smallskip
\caption{Error $\|p_{h,\tau}^\epsilon(t^k) - p_{h,\tau}^0(t^k)\|^2 + \sum_{j=1}^k \underline a \|m_{h,\tau}^\epsilon(t^j) - m_{h,\tau}^0(t^j)\|^2 = O(\epsilon^\alpha)$ between the discrete approximations for the hyperbolic problem and the parabolic limit problem for time $t^k=1$ and different values of $\epsilon$ and the discretization parameters $h$ and $\tau$.}
 \label{table:r3,2}
\end{center}
\end{table}
Again, the observations are in perfect agreement with the theoretical predictions made in Theorem~\ref{thm:r3}. 

Let us finally note that the previous formulas reveal that the error 
between the solutions of the hyperbolic and the parabolic problem actually behaves like 
\begin{eqnarray} \nonumber
\|p^\epsilon(t) - p^0(t)\|^2 + \|u^\epsilon(t) - u^0(t)\|^2 = O(\epsilon^4) \qquad  \mbox{for } t \gg \epsilon.
\end{eqnarray}
This shows that the estimate of Theorem~\ref{thm:r3} is dominated by the error in the mass flux within the initial layer $ 0 \le t  \preceq \epsilon$ which again resembles the fact that the 
second initial condition gets superfluous in the parabolic limit. 
This behavior can also be observed for the numerical approximations obtained with the method discussed in Section~\ref{sec:3}. A theoretical explanation of this fact would require a refined analysis which is left for future research.

\begin{acknowledgement}
The authors would like to gratefully acknowledge financial support by the German Research Foundation (DFG) via grants IRTG~1529, GSC~233, and TRR~154.
\end{acknowledgement}

\end{document}